\documentclass[runningheads]{llncs}
\usepackage[T1]{fontenc}
\usepackage{amsmath}
\usepackage{amssymb}

\newcommand{\norm}[1]{\left\lVert #1 \right\rVert}

\usepackage{epsfig}
\usepackage{subfig}
\usepackage{graphicx}
\begin{document}
	\title{Fractal Based Rational Cubic Trigonometric Zipper Interpolation with Positivity Constraints
	}
	\titlerunning{Positivity Preserving RCTZFIF}
	%
	\author{A. K. Sharma\inst{1}\orcidID{0009-0002-5562-4331} \and
		K. R. Tyada\inst{2}\orcidID{0000-0002-3780-0357}}
	\authorrunning{Sharma and Tyada}
	%
	\institute{Department of Mathematics and Computer Science,\\ Sri Sathya Sai Institute of Higher Learning,\\ Prasanthi Nilayam, Puttaparthi, Andhra Pradesh, India - 515134
		\email{abhishekkumars@sssihl.edu.in}\\ \and
		Department of Mathematics and Computer Science,\\ Sri Sathya Sai Institute of Higher Learning,\\ Prasanthi Nilayam, Puttaparthi, Andhra Pradesh, India - 515134\\
		\email{kurmaraotyada@sssihl.edu.in}}
	\maketitle              
	\begin{abstract}
		We propose a novel fractal based interpolation scheme termed Rational Cubic Trigonometric Zipper Fractal Interpolation Functions (RCTZFIFs) designed to model and preserve the inherent geometric property, positivity, in given univariate dataset. The method employs a combination of rational cubic trigonometric functions within a zipper fractal framework, offering enhanced flexibility through shape parameters and scaling factors. Rigorous error analysis is presented to establish the convergence of the proposed zipper fractal interpolants to the underlying classical fractal functions, and subsequently, to the data-generating function. We derive necessary constraints on the scaling factors and shape parameters to ensure positivity preservation. By carefully selecting the signature, shape parameters, and scaling factors within these bounds, we construct a class of RCTZFIFs that effectively preserve the positive nature of the data. Numerical experiments and visualisations demonstrate the efficacy and robustness of our approach in preserving positivity while offering fractal flexibility.
		
	\end{abstract}
	
	\section{Introduction}
	In the fields of science and engineering, various kinds of data points are obtained from sampling or experimentation. Finding the values in the intermediate independent points and visualising the data becomes very important. However, it is not sufficient for the interpolant to merely fit the data; it should also be capable of preserving the intrinsic geometric aspects of the data, such as convexity, monotonicity and positivity. For example, concentrations of substances in chemistry experiments are inherently positive, and any interpolant that fails to preserve positivity can produce physically meaningless results. Shape preservation is an important characteristic of fractal interpolation. A considerable amount of research can be found on shape preserving classical interpolants. In \cite{AMA2013}, Abbas presented a positivity preserving $\mathcal{C}^2-$piecewise rational cubic interpolation scheme. Sarfaraz~\cite{sarfraz2012shape}, introduced a piecewise rational cubic interpolation technique specifically designed for shape-preserving of 2D data, where each interval consists of four shape parameters- two of which are used for maintaining the shape characteristics while the other two offers flexibility as to enhance the curve. Monotonicity and convexity preserving $\mathcal{C}^1$-quadratic splines were developed by Schumaker~\cite{schumaker1983shape}.
	
	\par Classical interpolation schemes, such as polynomial, spline and rational methods, produce smooth interpolants that are differentiable everywhere except for some finite number of points~\cite{navascues2014fractal}. However, in many real world applications, the data is far from smooth. They exhibit irregularities, oscillations and complex patterns that are not well captured by the classical interpolants. To model such data points, it becomes essential to develop an interpolation technique that is capable of mimicking the nonsmooth nature observed in practice. This led to the consideration of fractal interpolation functions (FIFs), which offer a powerful framework for construction of interpolants with controlled irregularity \cite{navascues2014fractal}. 
	
	\par In \cite{Barnsley86}, Barnsley introduced fractal interpolation functions (FIFs), which are constructed using iterated function systems (IFS). The graph of an FIF corresponds to the attractor of the associated IFS. Furthermore, FIFs can be characterised as fixed points of the Read-Bajraktarević operator defined on an appropriate function space~\cite{BarHar86}. A vast amount of research has been done in FIFs owing to their utility in various domains. Barnsley in ~\cite{BEH89} introduced hidden variables fractal interpolation functions to provide more flexibility to the FIFs. The values of the interpolants depends on the "hidden variable" and potentially gives a better approximation. Chand and Navascués~\cite{CN2009}, extended the classical Hermite interpolation through a class of fractal interpolants. Inspired by the work of Barnsley, M.Navascués~\cite{Navascues2005} defined the $\alpha-$fractal interpolation functions, which not only interpolated but also approximated any continuous functions defined over the interval in $\mathbb{R}$.
	
	\par  In ~\cite{AMA2012}, a rational cubic monotonicity preserving fractal interpolation function was developed. Chand and Tyada~\cite{CK2018} studied the interpolation and approximation of constrained data. Rational interpolation functions provide proficient shape parameters to preserve the shape of the data. A $\mathcal{C}^2$-rational quintic fractal interpolation function that ensures monotonicity and positivity through suitable constraints on shape and scaling parameters was developed by Tyada~\cite{KC2018}. In~\cite{CK2015}, the shape preserving rational cubic FIF was developed for interpolation data lying between two piecewise straight lines. But such interpolation techniques cannot represent curves like circular arc, periodic data points, etc. To solve this problem, trigonometric basis functions or a blend of trigonometric and polynomial functions are used. For example, in ~\cite{chand2015positivity}, a rational cubic trigonometric FIF was developed with conditions for preserving the positivity of the given data points. K.R.Tyada and Chand in ~\cite{Tyada2021}, introduced a smooth RCTFIF in literature for the interpolation of periodic data. They studied the constrained interpolation, which is sufficient for preserving data positivity.
	
	\par In recent years, zippers have been used in fractal interpolation functions. The presence of a binary vector called a signature gives more flexibility to the FIFs. Aseev et al. ~\cite{Aseev2003}, introduced the zipper by generalising the IFS using a binary vector called the zipper. Chand et al in \cite{Chand2020} constructed the affine zipper fractal interpolation function and derived the basis for the space of these interpolants. They also showed how zipper fractal interpolation performed better than the classical fractal interpolation function. Zipper fractal interpolation functions with variable scalings were studied by Vijay and Chand in ~\cite{Chand2022}. In ~\cite{Jha2020}, the zipper quadratic fractal interpolation function was proposed with positivity preserving conditions.
	
	\par In this paper, we develop a novel method of rational cubcic trigonometric zipper fractal interpolation functions (RCTZFIFs) to preserve the positivity nature of the univariate periodic data set. The paper is organised as follows: Section \ref{sec:prerequisite} provides essential preliminaries on Zipper Fractal Interpolation Functions (ZFIFs), and the foundational framework for constructing ZFIFs. Section \ref{sec:construction} presents the detailed construction methodology for the proposed RCTZFIF. In Section \ref{sec:convergence}, we perform a convergence analysis of the interpolant. Section \ref{sec:positivity} discusses the positivity preserving property of RCTZFIFs, where we derive the conditions for the shape parameters and scaling factors. Section \ref{sec:numexp} provides a numerical illustrations of the theoretical results and effectiveness of the proposed scheme in preserving positivity. Finally, the paper concludes with a summary of key findings and outlines directions for future research.
	
	\section{Preliminaries of Fractal Interpolation Functions}
	\label{sec:prerequisite}
	We begin our exploration by laying the groundwork for fractal interpolation functions which are constructed based on the iterated function system.

	Let $\mathbb{N}_n = \{1,2,\dots,n\}$, and $\mathbb{N}_{n-1} = \{1,2,\dots,n-1\}$. Suppose we are given a strictly increasing sequence of real numbers $t_1 < t_2 < \cdots < t_n$ with associated data values $\{(t_i, f_i)\}_{i=1}^n \subset \mathbb{R}^2$. Define the global interval $I = [t_1, t_n]$ and subintervals $I_j = [t_j, t_{j+1}]$ for every $j \in \mathbb{N}_{n-1}$.
	
	For each subinterval, let $L_j : I \rightarrow I_j$ be a contraction mapping satisfying
	\[
	L_j(t_1) = t_j, \ L_j(t_n) = t_{j+1}, \ \text{and} \ |L_j(t) - L_j(t^*)| \leq \ell_j |t - t^*|, \ \ell_j \in [0,1).
	\]
	
	Next, define a family of vertical mappings $F_j : I \times \mathbb{R} \rightarrow \mathbb{R}$ satisfying
	\[
	|F_j(t, f) - F_j(t, f^*)| \leq |\lambda_j| |f - f^*|, \ \forall\, t \in I,\ f, f^* \in \mathbb{R}, \ \lambda_j \in (-1,1),
	\]
	and interpolating endpoint conditions:
	\begin{equation}\label{iterp_condi_2}
		F_j(t_1, f_1) = f_j, \ F_j(t_n, f_n) = f_{j+1}.
	\end{equation}
	
	We now define the iterated function system (IFS) maps
	\[
	W_j(t, f) = (L_j(t), F_j(t, f)) : I \times \mathbb{R} \to I_j \times \mathbb{R}, \ j \in \mathbb{N}_{n-1}.
	\]
	
	\begin{theorem}[Existence of FIF, cf. \cite{BarHar86}]
		The IFS $\{W_j\}_{j=1}^{n-1}$ admits a unique attractor $G$, which coincides with the graph of a continuous function $h^* : I \rightarrow \mathbb{R}$ that interpolates the given data, i.e., $h^*(t_i) = f_i$ for each $i \in N_n$. Moreover, let $\mathcal{C}(I) = \{h \in C(I) : h(t_1) = f_1,\ h(t_n) = f_n\}$ be the space of continuous functions satisfying the boundary interpolation conditions, equipped with the uniform norm. Define an operator $T$ on $\mathcal{C}(I)$ by
		$$ (Th)(t) := F_j\left(L_j^{-1}(t),\ h \circ L_j^{-1}(t)\right), \ t \in I_j.$$
		Then $T$ has a unique fixed point $h^*$ in $\mathcal{C}(I)$, referred to as the fractal interpolation function (FIF).
	\end{theorem}
	The construction of the FIF is based on an IFS composed of affine maps of the form:
	\begin{equation}
		L_j(t) = a_j t + e_j, \ F_j(t, f) = \lambda_j f + M_j(t), \ \text{for } j \in \mathbb{N}_{n-1}.
	\end{equation}
	
	In this setup, each $M_j : I \rightarrow \mathbb{R}$ is a continuous auxiliary function chosen to satisfy the interpolation constraints (\ref{iterp_condi_2}). The constants $\lambda_j$ are known as vertical scaling factors of the IFS, and they play a pivotal role in shaping the self-referential geometry of the interpolant. The use of a scaling vector adds an important degree of flexibility to FIFs, which makes them more versatile than classical interpolation methods. This extra flexibility helps adjust the shape of the interpolant and allows better control over smoothness and local geometric features, which is especially useful for preserving desired shapes in the given data.
	
	The existence of spline FIFs was first established by Barnsley~\cite{BarHar86}, using ideas from the calculus of fractal functions. They showed how the construct of smooth interpolants that pass through given data points can be achieved using a recursive method. Later, this idea was implemented for rational spline FIFs, as given in Theorem~\ref{existence_of_rationalFIF} (see also~\cite{CK2015}), which further improved the ability to build smooth and shape preserving fractal interpolants.
	
	\begin{theorem}[Existence of Rational Spline FIFs {\cite{CK2015}}]\label{existence_of_rationalFIF}
		Let $\{(t_i, f_i) : i \in \mathbb{N}_{n-1}\}$ be a strictly increasing data set. Define $L_j(t) = a_j t + b_j$ with 
		$a_j = \frac{t_{j+1} - t_j}{t_n - t_1}$ and $b_j = \frac{t_n t_j - t_1 t_{j+1}}{t_n - t_1}$, for all $j \in N_{n-1}$. Let $F_j(t, f) = \lambda_j f + M_j(t)$, where $M_j(t) = \frac{p_j(t)}{q_j(t)}$, with $\deg(p_j) = r$, $\deg(q_j) = s$, and $q_j(t) \neq 0$ on $[t_1, t_n]$. Suppose $|\lambda_j| < a_j^p$ for some integer $p \geq 0$. For $m = 1, 2, \dots, p$, define
		\[
		F_{j,m}(t, f) = \frac{\lambda_j f + M_j^{(m)}(t)}{a_j^m}, \ 
		f_{1,m} = \frac{M_1^{(m)}(t_1)}{a_1^m - \lambda_1},\
		f_{n,m} = \frac{M_{n-1}^{(m)}(t_n)}{a_{n-1}^m - \lambda_{n-1}}.
		\]
		If $F_{k,m}(t_n, f_{n,m}) = F_{k+1,m}(t_1, f_{1,m})$ ,$\forall k = 1, \dots, n-2$ and $m = 1, \dots, p$, then the IFS $\{\omega_i(t,f) = (L_j(t), F_j(t,f))\}$ determines a rational FIF $\Phi \in \mathcal{C}^p[t_1, t_n]$.
	\end{theorem}

	\section{Construction of Rational Cubic Trigonometric Zipper Fractal Interpolation Functions (RCTZFIFs)} \label{sec:construction}
	In classical splines, each segment maps the main interval to smaller parts in the same direction using positive scaling. But if we also allow reverse mappings (negative scaling), we can build more flexible curves. To handle this, we use a binary signature vector that tells whether each segment follows the normal or reversed direction like a zipper joining alternating parts. This idea leads to the construction of zipper rational cubic splines (ZRCSs), which offer greater control and flexibility in shaping curves. In this section, we emphasise the construction of RCTZFIF. 
	
	Let $\{(t_i, f_i) \in I\times K : i \in \mathbb{N}_n\}$ be the given data points of the original function $\psi$ such that $t_i < t_{i+1}, \forall i \in \mathbb{N}_{n-1}$. Let us consider the signature $\epsilon = (\epsilon_1, \epsilon_2, \dots, \epsilon_{n-1}) \in \{0,1\}^{n-1}$. Let $I = [t_1, t_n]$ and $I_j = [t_j, t_{j+1}] , j \in \mathbb{N}_{n-1}$. Let $L_j : I \rightarrow I_j$, be the affine map given by $L_{j,\epsilon}(t) = a_j t + b_j, j \in \mathbb{N}_{n-1}$ such that 
	
	\begin{equation*}
		L_{j,\epsilon}(t_1) = t_{j+\epsilon_j}, L_j(t_n) = t_{j+1-\epsilon_j},
	\end{equation*}
	where $a_j = \frac{t_{j+1-\epsilon_j}-t_{j+\epsilon_j}}{t_n - t_1}$, and $b_j = \frac{t_n t_{j+\epsilon_j}-t_1 t_{j+1-\epsilon_j}}{t_n - t_1}$. Let $F_{j,\epsilon} : I \times K \rightarrow \mathbb{R}$ be a function defined as 
	\begin{equation}\label{eq_Fj}
		F_{j,\epsilon}(t, f) = \lambda_j + M_{j,\epsilon}(t),
	\end{equation}
	
	\noindent where $M_{j,\epsilon} : I \rightarrow \mathbb{R}$ is a rational function such that $
	M_{j,\epsilon}(t) = p_{j,\epsilon}(t)/q_{j,\epsilon}(t)$, where $p_{j,\epsilon}(t)$ and $q_{j,\epsilon}(t)$ are cubic trigonometric polynomials such that $q_{j,\epsilon}(t) \neq 0, \forall t\in[t_1, t_n]$ and $|\lambda_j| < a_j, j\in \mathbb{N}_{n-1}$ is the vertical scaling factor. 
	
	\noindent Let $F_{j,\epsilon}^{(1)}(t,d) = \frac{\lambda_j d + M_{j,\epsilon}^{(1)}(t)}{a_j}$, where, $M_{j,\epsilon}^{(1)}(t)$ is the first order derivative of $M_{j,\epsilon}(t), t\in[t_1, t_n], j \in \mathbb{N}_{n-1}$. $F_{j,\epsilon}(t, f)$ and $F_{j,\epsilon}^{(1)}(t,d)$ satisfy the following:
	\begin{align*}
		F_{j,\epsilon}(t_1, f_1) = f_{j+\epsilon_j}, F_{j,\epsilon}(t_n, f_n) = f_{j+1-\epsilon_j}, F_{j,\epsilon}^{(1)}(t_1, d_1) = d_{j+\epsilon_j}, F_{j,\epsilon}^{(1)}(t_n, d_n) = d_{j+1-\epsilon_j}.
	\end{align*}
	
	\noindent By Theorem \ref{existence_of_rationalFIF}, the fixed point of the IFS $\{I\times K; (L_{j,\epsilon}(t), F_{j,\epsilon}(t,f)), j\in \mathbb{N}_{n-1}\}$ will be a graph of a $\mathcal{C}^1-$rational cubic trigonometric Fractal Interpolation Function. As per  Equation (\ref{eq_Fj}), the RCTZFIF will be defined as 
	\begin{equation}\label{eq_phie}
		\phi_\epsilon(L_j(t)) = \lambda_j \phi_\epsilon(t) + M_{j,\epsilon}(t),
	\end{equation}
	where $M_{j,\epsilon}(t) = \frac{p_{j,\epsilon}(\theta)}{q_{j,\epsilon}(\theta)}, \theta = \frac{\pi}{2}\frac{t-t_1}{t_n - t_1}$, $
		p_{j,\epsilon}(\theta) = (1-\sin\theta)^3 U_{j,\epsilon} + \sin\theta(1-\sin\theta)^2V_{j,\epsilon} + \cos\theta(1-\cos\theta)^2W_{j,\epsilon} + (1-\cos\theta)^3X_{j,\epsilon}$, and $q_{j, \epsilon}(\theta)  =  (1-\sin\theta)^3 \alpha_{j,\epsilon} + \sin\theta(1-\sin\theta)^2\beta_{j,\epsilon} + \cos\theta(1-\cos\theta)^2\gamma_{j,\epsilon} + (1-\cos\theta)^3\delta_{j,\epsilon}$. Furthermore, $\phi_\epsilon$ satisfies the following $\mathcal{C}^1$-continuity conditions:
	\begin{align*}
		\phi_\epsilon(L_j(t_1)) = f_{j+\epsilon_j}, \ \phi_\epsilon(L_j(t_n)) = f_{j+1-\epsilon_j}, \  \phi_\epsilon'(L_j(t_1)) = d_{j+\epsilon_j}, \
		\phi_\epsilon'(L_j(t_n)) = d_{j+1-\epsilon_j}.
	\end{align*}
	
	\noindent At $t=t_1$, $\theta = 0$. Then $\phi_\epsilon(L_j(t_1)) = \lambda_jf_1 + M_{j,\epsilon}(t_1) = \lambda_jf_1 + \frac{U_{j, \epsilon}}{\alpha_{j, \epsilon}} = f_{j+\epsilon_j}$. Therefore, we get 
	\[
	U_{j,\epsilon} = \begin{cases}
		\alpha_{j,\epsilon}(f_j-\lambda_jf_1), & \epsilon_j = 0\\
		\alpha_{j,\epsilon}(f_{j+1}-\lambda_jf_1), & \epsilon_j = 1.
	\end{cases}
	\]
	At $t=t_n$, $\theta = \pi/2$. This implies
	$\phi_\epsilon(L_i(t_n))=\lambda_j\phi_\epsilon(t_n) + M_{j,\epsilon}(t_n) = \lambda_jf_n + \frac{X_{j,\epsilon}}{\delta_{j,\epsilon}} = f_{j+1-\epsilon_j}$. Therefore, 
	\[
	X_{j,\epsilon} = \begin{cases}
		\delta_{j,\epsilon}(f_{j+1} - \lambda_jf_n), & \epsilon_j = 0\\
		\delta_{j,\epsilon}(f_{j} - \lambda_jf_n), & \epsilon_j = 1.
	\end{cases}
	\]
	Similarly, at $t=t_1$, $\phi_\epsilon'(L_j(t_1)) = d_{j+\epsilon_j}$. 
	Let $d_j^* = a_jd_{j+\epsilon_j}-\lambda_jd_1$ then, 
	\[
	V_{j,\epsilon} = \begin{cases}
		\beta_{j,\epsilon}(f_j-\lambda_jf_1)+\frac{2(t_n-t_1)\alpha_{j,\epsilon}(a_jd_j-\lambda_jd_1)}{\pi}, & \epsilon_j = 0\\
		\beta_{j,\epsilon}(f_{j+1}-\lambda_jf_1)+\frac{2(t_n-t_1)\alpha_{j,\epsilon}(a_jd_{j+1}-\lambda_jd_1)}{\pi}, & \epsilon_j = 1.
	\end{cases}
	\]
	Finally, when $t=t_n$ we get
	\[
	W_{j,\epsilon} = \begin{cases}
		\gamma_{j,\epsilon}(f_{i+1} - \lambda_j f_n) - \frac{2(t_n - t_1)\delta_{j,\epsilon}(a_jd_{j+1} - \lambda_jd_n)}{\pi}, & \epsilon = 0\\
		\gamma_{j,\epsilon}(f_{i} - \lambda_j f_n) - \frac{2(t_n - t_1)\delta_{j,\epsilon}(a_jd_{j} - \lambda_jd_n)}{\pi}, & \epsilon = 1.
	\end{cases}
	\]
	
	\noindent By substituting $U_{j,\epsilon}, V_{j,\epsilon}, W_{j,\epsilon}, \text{ and } X_{j,\epsilon}$ in the equation (\ref{eq_phie}), we get the required well-defined $\mathcal{C}^1$-rational cubic trigonometric zipper fractal interpolation function. For computing the derivatives from the given data, we have used arithmetic mean method (AMM).

	\section{Convergence Analysis}
	\label{sec:convergence}
	In this section, we analyse the convergence behaviour of the rational cubic trigonometric zipper fractal interpolation function (RCTZFIF) $\phi_\epsilon$ toward the original data-generating function $\psi$. To do so, we establish an upper bound on the approximation error $\|\phi_\epsilon - \psi\|_\infty$ using the triangle inequality:
	\begin{align}
		\|\phi_\epsilon - \psi\|_\infty \leq \|\phi_\epsilon - \phi\|_\infty + \|\phi - \psi\|_\infty,
	\end{align}
	where $\phi$ denotes the rational cubic trigonometric FIF. \cite{chand2015positivity}. The first term $\|\phi_\epsilon - \phi\|_\infty $ gives the error bound between RCTZFIF and RCTFIF, while the second term accounts for the error between the classical spline relative to the target function $\psi$, which is the RCTFIF.
	
	We now state the following result from~\cite{chand2015positivity}, which gives the error bound for the classical spline interpolation $\psi$ and the RCTFIF $\phi$:
	\begin{theorem}
		\label{thm:rctfif}
		(\cite{chand2015positivity}) Let $\psi \in C^3[t_1, t_n]$ be the original function that generates $\{(t_i, f_i),  i \in \mathbb{N}_{n}\}$ and let $\phi$ be the RCTFIF $\in \mathcal{C}^1[t_1, t_n]$. Let $d_i, i\in \mathbb{N}_n$ be the bounded first-order derivative at knot $t_i, i \in \mathbb{N}_n$. Let $|\lambda|_\infty = max\{|\lambda_j|,  j \in \mathbb{N}_{n-1}\}$ and the shape parameters $\alpha_j, \beta_j, \gamma_j, \delta_j,  j \in \mathbb{N}_{n-1}$ are non-negative with $\beta_j \geq\alpha_j, \gamma_j \geq \delta_j$. Then 
		\begin{align}
			||\psi - \phi||_\infty \leq \frac{1}{2}||\psi^{(3)}||_\infty h^3 c + \frac{|\lambda|_\infty}{1-|\lambda|_\infty}(E(h) + E^*(h)),
		\end{align}
		where $E(h) = ||\psi||_\infty + \frac{4h}{\pi}E_1 , E^*(h) = F + \frac{4h}{\pi}E_2, E_1 = \max\limits_{1 \le j \le n-1} \{d_j\}, F = max\{|f_1|, |f_n|\}$,\ $E_2 = max\{|d_1|, |d_n|\}$ and c is as defined in Proposition 2 of \cite{chand2015positivity}.
	\end{theorem}
	
	\noindent Now, to approximate the upper bound for $||\phi_\epsilon - \phi||_\infty$, consider
	\begin{align*}
		||\phi_\epsilon - \phi||_\infty = ||(\lambda_j \phi_\epsilon - \lambda_j \phi)+ M_{j,\epsilon} - M_j||_\infty \leq |\lambda|_\infty || \phi_\epsilon - \phi||_\infty + ||M_{j,\epsilon} - M_j||_\infty.
	\end{align*}
	Therefore, we get $||\phi_\epsilon - \phi||_\infty \leq \frac{||M_{j,\epsilon} - M_j||_\infty}{1-|\lambda|_\infty}$. Now
	\[
	\begin{aligned}
		\left\lVert M_{j,\epsilon} - M_j \right\rVert_\infty 
		= & \left\lVert \frac{\alpha_{j,\epsilon} B_0}{q_{j,\epsilon}(\theta)}[f_{j+\epsilon_j} - f_j] 
		+ \frac{\beta_{j,\epsilon} B_1}{q_{j,\epsilon}(\theta)}[f_{j+\epsilon_j} - f_j] 
		+ \frac{\gamma_{j,\epsilon} B_2}{q_{j,\epsilon}(\theta)}[f_{j+1-\epsilon_i} - f_{j+1}] \right. \\
		&\ \left. + \frac{\delta_{j,\epsilon} B_3}{q_{j,\epsilon}(\theta)}[f_{j+1-\epsilon_j} - f_{j+1}]
		+ \frac{2l\alpha_{j,\epsilon} a_j B_1}{\pi q_{j,\epsilon}(\theta)}[d_{j+\epsilon_j} - d_j] \right.\\
		& \left. - \frac{2l\delta_{j,\epsilon} a_j B_2}{\pi q_{j,\epsilon}(\theta)}[d_{j+1-\epsilon_j} - d_{j+1}], \right\rVert_\infty,
	\end{aligned}
	\]
	\noindent where $B_0 = (1-\sin\theta)^3, B_1 = \sin\theta(1-\sin\theta)^2, B_2 = \cos(\theta)(1-\cos(\theta))^2, B_3 = (1-\cos(\theta))^3, \theta = \frac{\pi}{2}\left(\frac{t-t_1}{l}\right), \text{and}\ l = t_n - t_1$. Assume that $\xi = \max\limits_{1 \le j \le n-1} |q_j(\theta)|,\ \norm{\epsilon}_\infty = \max\limits_{1 \le j \le n-1} |\epsilon_j|$. Then
	\begin{align*}
		\left\lVert M_{j,\epsilon} - M_j \right\rVert_\infty 
		&\leq \frac{\alpha_{j,\epsilon} B_0}{\xi} \norm{f_{j+\epsilon_j} - f_j}_\infty 
		+ \frac{\beta_{j,\epsilon} B_1}{\xi} \norm{f_{j+\epsilon_j} - f_j}_\infty \\
		&\ + \frac{\gamma_{j,\epsilon} B_2}{\xi} \norm{f_{j+1 - \epsilon_j} - f_{j+1}}_\infty 
		+ \frac{\delta_{j,\epsilon} B_3}{\xi} \norm{f_{j+1 - \epsilon_j} - f_{j+1}}_\infty \\
		&\ + \frac{2l \alpha_{j,\epsilon} a_j B_1}{\pi \xi} \norm{d_{j + \epsilon_j} - d_j}_\infty 
		- \frac{2l \delta_{j,\epsilon} a_j B_2}{\pi \xi} \norm{d_{j+1 - \epsilon_j} - d_{j+1}}_\infty\\
		&\leq \frac{\alpha_{j,\epsilon} B_0+\beta_{j,\epsilon} B_1+\gamma_{j,\epsilon} B_2+\delta_{j,\epsilon} B_3}{\xi}\norm{f_{j+\epsilon_j} - f_j}_\infty \norm{\epsilon}_\infty \\
		&\ + \frac{2la_j}{\pi\xi}\norm{d_{j+1}-d_j}_\infty \norm{\epsilon}_\infty[\alpha_{j,\epsilon}B_1 - \delta_{j,\epsilon}B_2].
	\end{align*}
	
	\noindent Since $h_j = la_j$, and by Theorem \ref{thm:rctfif}, we get $\norm{M_{j,\epsilon} - M_j}_\infty \leq \norm{\epsilon}_\infty \left(C\Phi + \frac{4h_j}{\pi}\eta \right)$, where $\Phi = \max\limits_{1 \le j \le n-1} |f_j|, \eta = \max\limits_{1 \le j \le n-1} |d_j|$. Hence $\norm{\phi_\epsilon - \phi} \leq \frac{\norm{\epsilon}_\infty}{1-|\lambda|_\infty}\left(C\Phi + \frac{4h_j}{\pi}\eta \right)$.
	
	\noindent Combining the above estimates, we get the following:  
	\begin{align*}
		\norm{\phi_\epsilon - \psi}_\infty &\leq \norm{\phi_\epsilon - \phi}_\infty + \norm{\phi - \psi}_\infty\\
		&\leq \frac{\norm{\epsilon}_\infty}{1-|\lambda|_\infty}\left(C\Phi + \frac{4h}{\pi}\eta \right) + 
		\frac{1}{2}\norm{\psi^{(3)}}_\infty h^3c + \frac{|\lambda|_\infty}{1-|\lambda|_\infty}\left(E(h) + E(h^*)\right),
	\end{align*}
	where $h = \displaystyle \mathop{\max}_{1 \leq j \leq n-1} |h_j|$, and $E(h) \text{ and } E^*(h)$ are as defined in Theorem \ref{thm:rctfif}.
	
	\noindent We summarize this whole process in the following theorem.
	\begin{theorem}
		\label{thm:convergence}
		Let $\phi_\epsilon$ be the $\mathcal{C}^1$-continuous RCTZFIF with corresponding signature $\epsilon = (\epsilon_1, \epsilon_2, \dots, \epsilon_{n-1}) \in \{0,1\}^{n-1}$. Let $\psi \in \mathcal{C}^3[a,b]$ be the data generating function, and let the generated data points be $\{(t_i, f_i),  i \in \mathbb{N}_{n}\}$. Assume that for each $j \in \mathbb{N}_{n-1}$, the first-order derivative $d_j$ at the knot $t_j$ is bounded. Define $|\lambda|_\infty := \max\{|\lambda_j| : j \in \mathbb{N}_{n-1}\}$. Suppose the shape parameters $\alpha_{j,\epsilon}, \beta_{j,\epsilon}, \gamma_{j,\epsilon}, \delta_{j,\epsilon}$ are all non-negative for each $j \in \mathbb{N}_{n-1}$, and satisfy the conditions $\beta_{j,\epsilon} \geq \alpha_{j,\epsilon}$ and $\gamma_{j,\epsilon} \geq \delta_{j,\epsilon}$. Then 
		\begin{align*}
			\norm{\phi_\epsilon - \psi}_\infty &\leq \norm{\phi_\epsilon - \phi}_\infty + \norm{\phi - \psi}_\infty\\
			&\leq \frac{\norm{\epsilon}_\infty}{1-|\lambda|_\infty}\left(C\Phi + \frac{4h}{\pi}\eta \right) + 
			\frac{1}{2}\norm{\psi^{(3)}}_\infty h^3c + \frac{|\lambda|_\infty}{1-|\lambda|_\infty}\left(E(h) + E(h^*)\right),
		\end{align*}
		where $\norm{\epsilon}_\infty = \max\limits_{1 \le j \le n-1} |\epsilon_j|, \
		\norm{\lambda}_\infty = \max\limits_{1 \le j \le n-1} |\lambda_j|,\ C = \frac{q_{j,\epsilon}(\theta)}{\xi}, \xi = \max\limits_{1 \le j \le n-1} |q_{j,\epsilon}(\theta)|, \forall j \in \mathbb{N}_{n-1}, \ \Phi = \max\limits_{1 \le j \le n-1} |f_j|$, and $\eta = \max\limits_{1 \le j \le n-1} |d_j|$. $E(h) \text{ and } E(h^*) \text{ are same as defined in Theorem }\ref{thm:rctfif}$.
	\end{theorem}
	
	\section{Positivity preserving RCTZFIFs}
	\label{sec:positivity}
	Preserving the shape of the interpolant, especially ensuring positivity, is essential in various scientific and engineering contexts, where negative values lack physical interpretation. This is particularly relevant in fields such as population modelling, chemical concentration analysis, and probability theory, where the quantities being interpolated are inherently nonnegative. This section is devoted to examining the criteria that ensure the constructed RCTZFIF maintains the positivity inherent in the given interpolation data.
	
	Let $\{(t_i, f_i) : i \in \mathbb{N}_n\}$ be the positive data set, where $t_1 < t_2 < \cdots < t_n$ and $f_i > 0$ for each $i \in \mathbb{N}_n$. Our objective is to construct an RCTZFIF $\phi_\epsilon$ such that $\phi_\epsilon(t) > 0$ for all $t \in [t_1, t_n]$. 
	
	To guarantee the preservation of positivity, we derive sufficient conditions on the scaling coefficients $\lambda_j$ and the shape parameters $\alpha_{j,\epsilon}, \beta_{j,\epsilon}, \gamma_{j,\epsilon} \text{ and } \delta_{j,\epsilon}$. These conditions are rigorously formulated in the theorem presented below.
	
	\begin{theorem} \label{thm:positivity}
		Let $\phi_\epsilon$ be the RCTZFIF which interpolates the given data 
		$\{(t_i, f_i),\\ i \in \mathbb{N}_n\}$. Let $\epsilon = (\epsilon_1, \epsilon_2, \dots, \epsilon_{n-1}) \in \{0,1\}^{n-1}$ be the signature of the RCTZFIF $\phi_\epsilon$. Let $f_i > 0$ for all $i \in \mathbb{N}_n$. Let $\alpha_{j,\epsilon}, \beta_{j,\epsilon}, \gamma_{j,\epsilon}, \text{ and } \delta_{j,\epsilon}$ are the shape parameters and $\lambda_j, j\in \mathbb{N}_{n-1}$, is the scaling factor of $\phi_\epsilon$ in the $j^{th}$ sub-interval. Then, $\phi_\epsilon$ 
		preserves the positivity of the data if:
		\begin{enumerate}
			\item The scaling factors are chosen such that: $0 \leq \lambda_j < \min\left\{ a_j, \frac{f_{j + \epsilon_j}}{f_1}, \frac{f_{j + 1 - \epsilon_j}}{f_n} \right\}$.
			\item The shape parameters satisfy the following: $\alpha_{j,\epsilon} > 0, \delta_{j,\epsilon} > 0,$
			\begin{equation*}
				\beta_{j,\epsilon} > \max\left\{ 0, \frac{-2l\alpha_{j,\epsilon} d_{j,\epsilon}^*}{\pi f_{j,\epsilon}^*} \right\},\ 
				\gamma_{j,\epsilon} > \max\left\{ 0, \frac{2l\delta_{j,\epsilon} d_{j+1,\epsilon}^*}{\pi f_{j+1, \epsilon}^*}\right\},
			\end{equation*}
			where $f_{j,\epsilon}^* = f_{j+\epsilon_j} - \lambda_j f_1$, $f_{j+1,\epsilon}^* = f_{j+1-\epsilon_j} - \lambda_j f_n$, $d_{j,\epsilon}^* = a_jd_{j+\epsilon_j} - \lambda_j d_1$, and $d_{j+1,\epsilon}^* = a_jd_{j+1-\epsilon_j} - \lambda_j d_n$.
		\end{enumerate}
	\end{theorem}
	\begin{proof}
		To ensure that the constructed $\mathcal{C}^1$-RCTZFIF $\phi_\epsilon$ remains positive over the domain $[t_1, t_n]$, it is sufficient to ensure that
		\begin{align}
			\phi_\epsilon(L_j(t)) > 0, \forall t \in [t_1, t_n], j\in \mathbb{N}_{n-1}.
		\end{align}
		From the referential equation~\eqref{eq_phie}, the above inequality is equivalent to
		\begin{align} \label{pos_con}
			\lambda_j \phi_\epsilon(t) + \frac{p_{j,\epsilon}(\theta)}{q_{j,\epsilon}(\theta)} > 0.
		\end{align}
		Assuming $\phi_\epsilon(t) \geq 0$, for all $t \in [t_1, t_n]$ and $0 \leq \lambda_j < 1$. It suffices to ensure that the rational term $\frac{p_{j,\epsilon}(\theta)}{q_{j,\epsilon}(\theta)}$ is strictly non-negative over the interval.
		
		Since $q_{j,\epsilon}(\theta)$ is constructed using non-negative shape parameters $\alpha_{j,\epsilon}, \beta_{j,\epsilon}, \gamma_{j,\epsilon},$ and $\delta_{j,\epsilon}$, it follows that $q_{j,\epsilon}(\theta) > 0$ for all $\theta \in [0, \pi/2]$. Thus, non-negativity of the numerator $p_{j,\epsilon}(\theta)$ ensures the desired result.
		
		Recall from Section~\ref{sec:construction} that $U_{j,\epsilon} = \alpha_{j,\epsilon}(f_{j+\epsilon_j} - \lambda_j f_1)$. So $U_{j,\epsilon} > 0$ whenever $\lambda_j < \frac{f_{j+\epsilon_j}}{f_1}$. Similarly, $X_{j,\epsilon} = \gamma_{j,\epsilon}(f_{j+1-\epsilon_j} - \lambda_j f_n) > 0 \ \text{if}\ \lambda_j < \frac{f_{j+1-\epsilon_j}}{f_n}$.
		
		For the term $V_{j,\epsilon} = \beta_{j,\epsilon}(f_{j+\epsilon_j} - \lambda_j f_1) + \frac{2l \alpha_{j,\epsilon} d_{j,\epsilon_j}^*}{\pi}$, we distinguish two cases. If $d_{j,\epsilon_j}^* \geq 0$, then $V_{j,\epsilon} > 0$ for any non-negative $\alpha_{j,\epsilon}$, $\beta_{j,\epsilon}$, and $\lambda_j < \frac{f_{j+\epsilon_j}}{f_1}$. Otherwise, if $d_{j,\epsilon_j}^* < 0$, $V_{j,\epsilon} > 0$ provided $\beta_{j,\epsilon} > \frac{-2l \alpha_{j,\epsilon} d_{j,\epsilon_j}^*}{\pi f_{j,\epsilon}^*}$. Likewise, for $W_{j,\epsilon} = \delta_{j,\epsilon}(f_{j+1-\epsilon_j} - \lambda_j f_n) - \frac{2l \gamma_{j,\epsilon} d_{j+1,\epsilon_j}^*}{\pi}$, positivity holds under the condition $\lambda_j < \frac{f_{j+1-\epsilon_j}}{f_n}$ when $d_{j+1,\epsilon_j}^* \leq 0$. If $d_{j+1,\epsilon_j}^* > 0$, then positivity requires $\gamma_{j,\epsilon} > \frac{2l \delta_{j,\epsilon} d_{j+1,\epsilon_j}^*}{\pi f_{j+1,\epsilon}^*}$. Thus, ensuring the above inequalities guarantees that $p_{j,\epsilon}(\theta) > 0$, and hence, by Equation~\eqref{pos_con}, the interpolant $\phi_\epsilon$ remains strictly positive on $[t_1, t_n]$.
	\end{proof}

\section{Numerical Experiments}
\label{sec:numexp}
We demonstrate the positivity-preserving nature of RCTZFIFs using the strictly positive univariate data set $\{(1,14), (3,2), (8, 0.8), (10, 0.65), (11, 0.75),\\ (12, 0.7), (16, 0.69)\}$. Figure~\ref{fig:rctzfif_all} illustrates various RCTZFIF configurations alongside classical interpolants for comparative analysis.

\begin{figure}[!htb]
	\centering
	\subfloat[Non-positive RCTZFIF $\phi_\epsilon^1$\label{fig:rctzfif_1a}]{\includegraphics[width=0.5\textwidth]{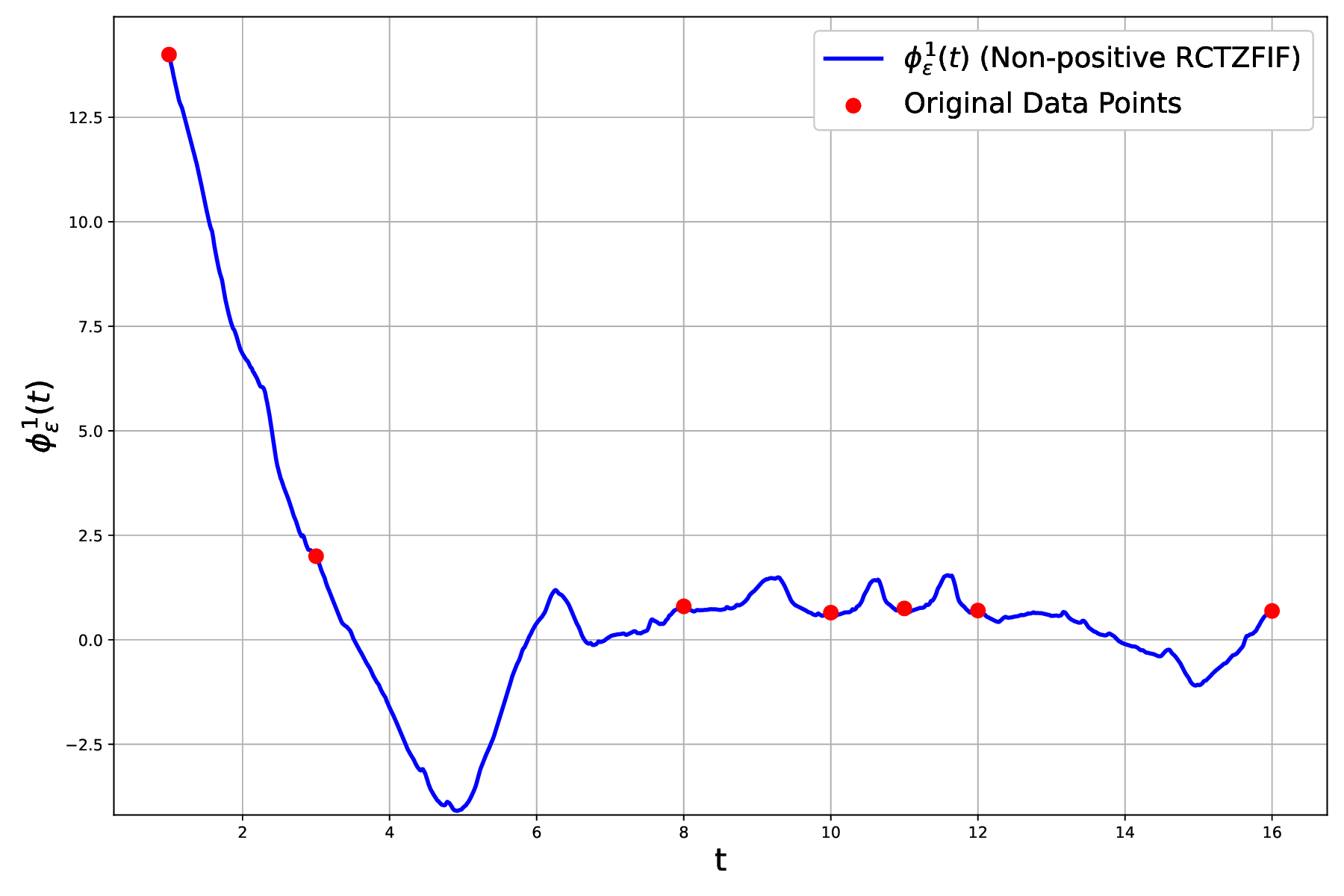}}\hfill
	\subfloat[Positive RCTZFIF $\phi_\epsilon^2$ \label{fig:rctzfif_1b}]{\includegraphics[width=0.5\textwidth]{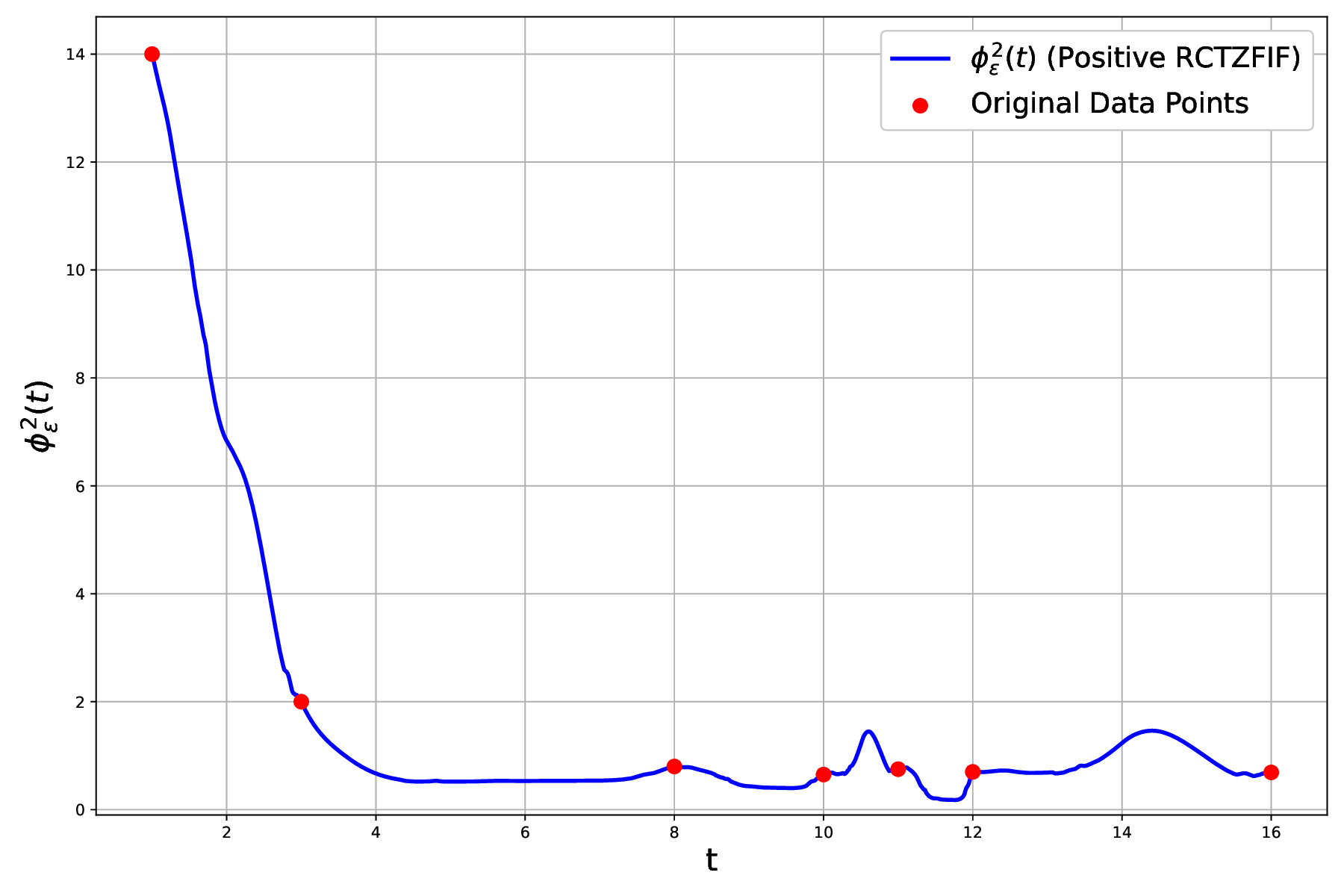}}\\
	\subfloat[Local perturbation effects of scaling parameters $\lambda_j$: $\phi_\epsilon^3$\label{fig:rctzfif_1c}]{\includegraphics[width=0.5\textwidth]{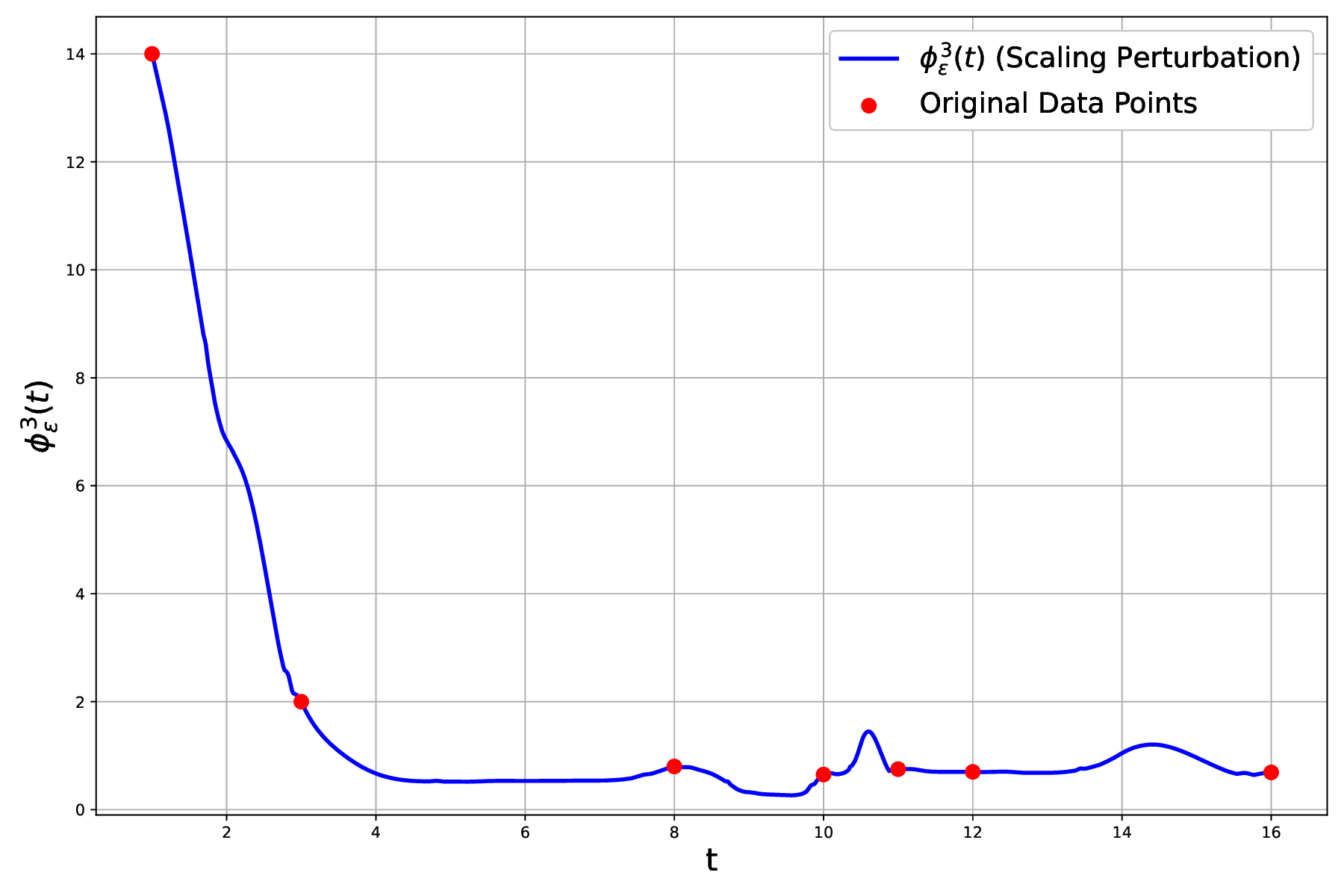}}\hfill
	\subfloat[Combined perturbation of $\lambda_j$, $\beta_j$, and $\gamma_j$: $\phi_\epsilon^4$\label{fig:rctzfif_1d}]{\includegraphics[width=0.5\textwidth]{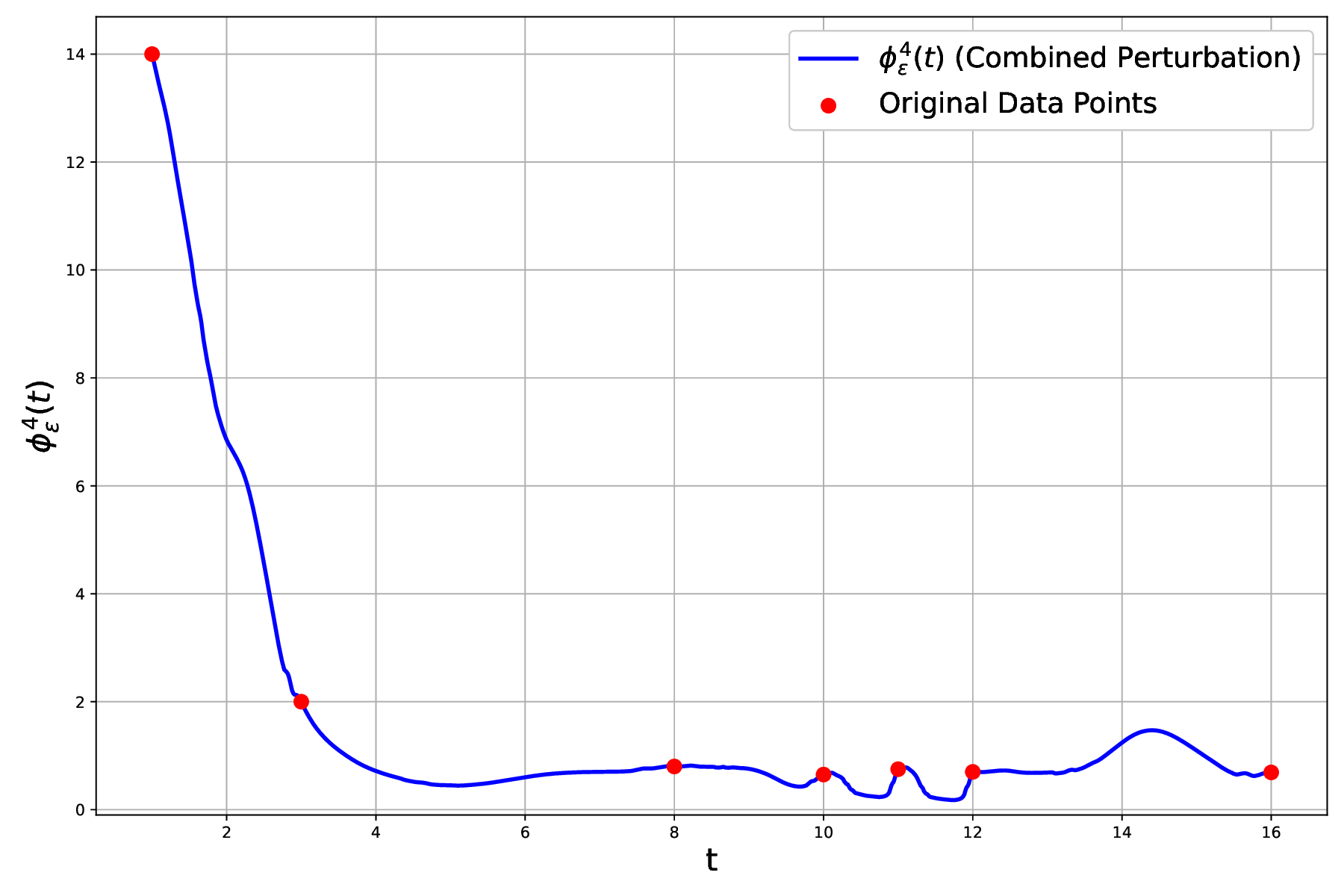}}\\
	\subfloat[Classical interpolant $\psi_\epsilon^1$ with $\epsilon = (0,0,0,0,0,0)$\label{fig:rctzfif_1e}]{\includegraphics[width=0.5\textwidth]{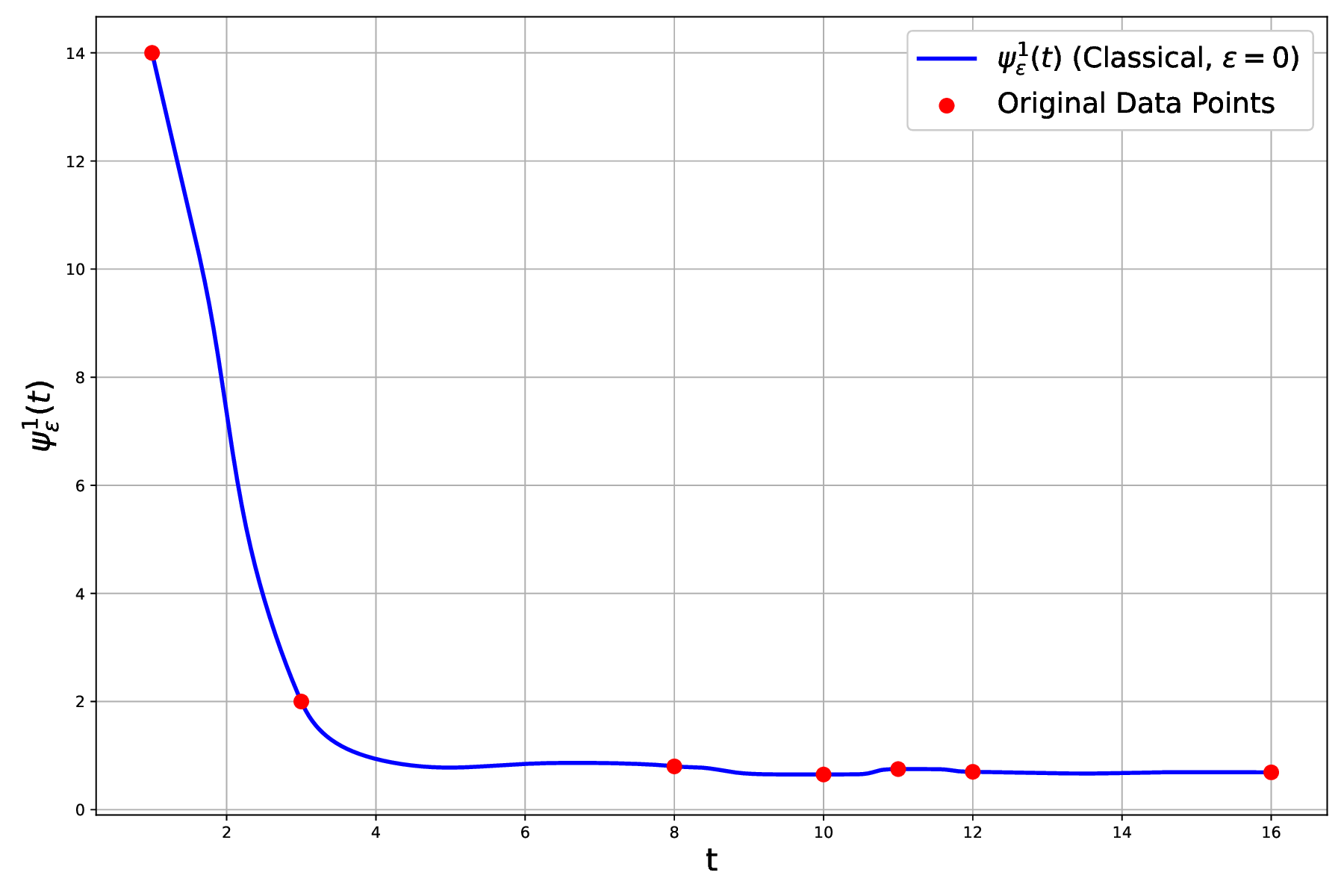}}\hfill
	\subfloat[Classical interpolant $\psi_\epsilon^2$ with $\epsilon = (1,1,1,1,1,1)$\label{fig:rctzfif_1f}]{\includegraphics[width=0.5\textwidth]{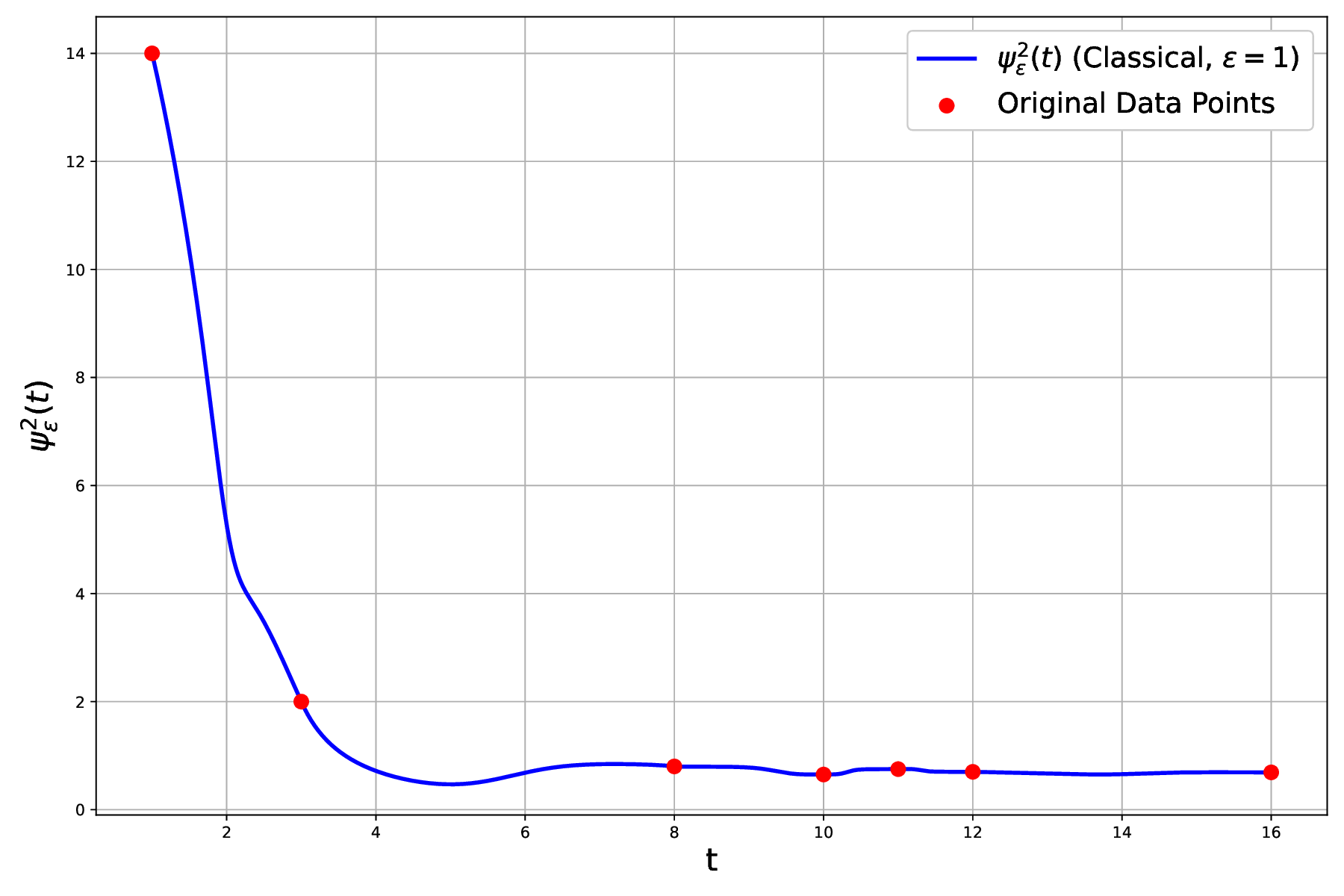}}
	\caption{Comparative visualization of RCTZFIFs and their classical counterparts RCTZIFs under different parameter configurations from Table \ref{table:rctzfif_params}.}
	\label{fig:rctzfif_all}
\end{figure}

\begin{figure}[!htb]
	\centering
	\subfloat[Second derivative of $\phi_\epsilon^1$\label{fig:rctzfif_2a}]{\includegraphics[width=0.5\textwidth]{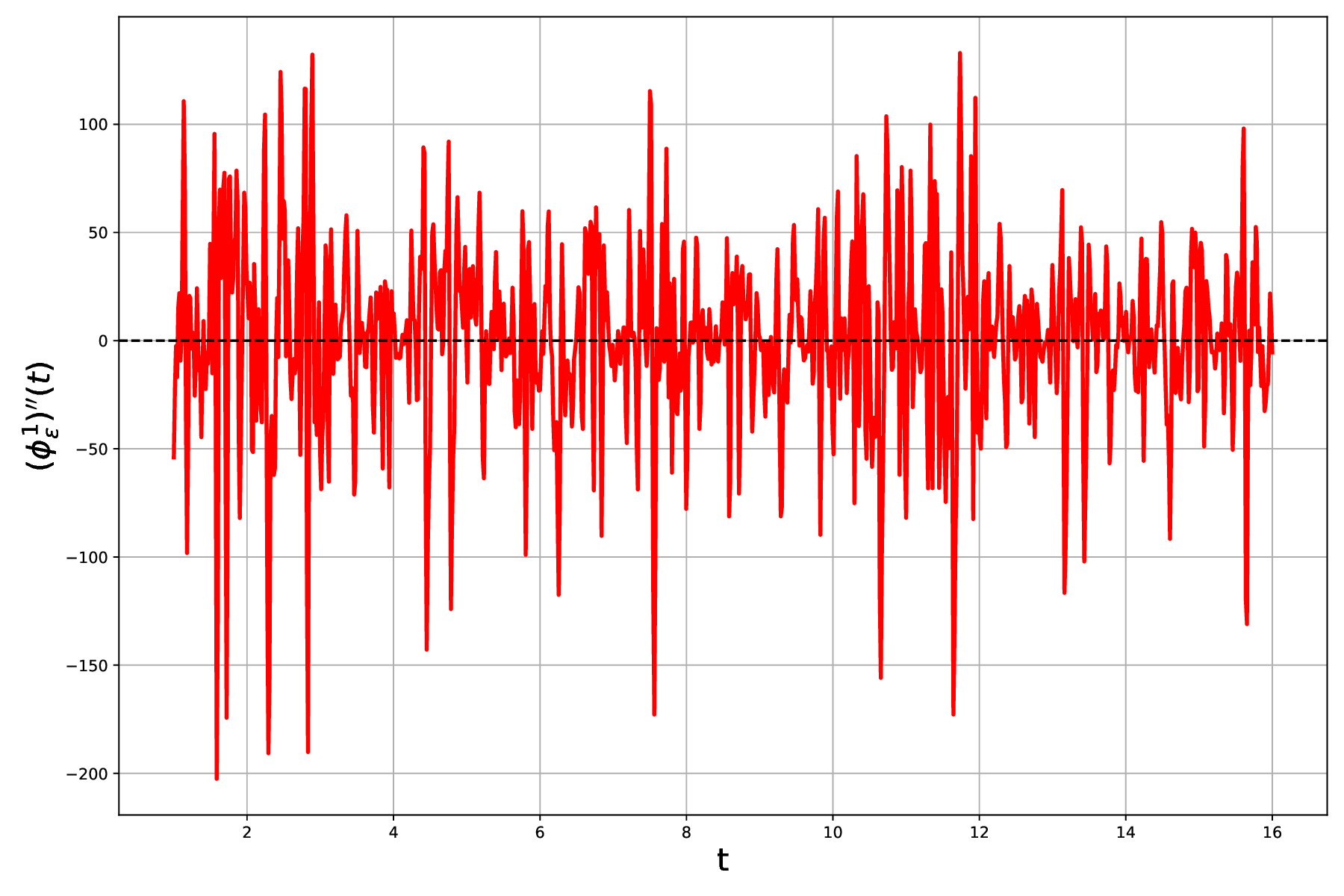}}\hfill
	\subfloat[Second derivative of $\phi_\epsilon^2$ \label{fig:rctzfif_2b}]{\includegraphics[width=0.5\textwidth]{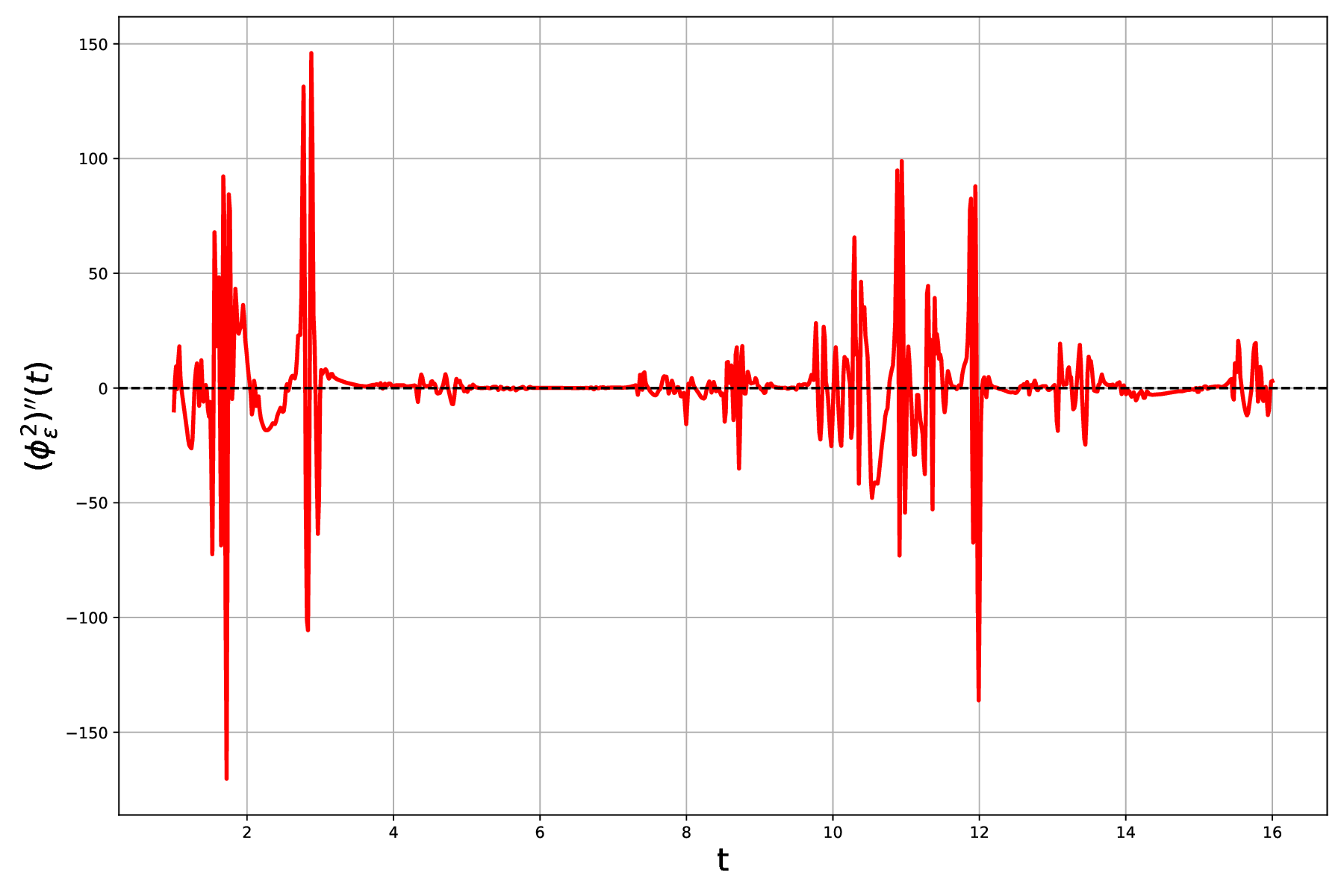}}\\
	\subfloat[Second derivative of $\phi_\epsilon^3$\label{fig:rctzfif_2c}]{\includegraphics[width=0.5\textwidth]{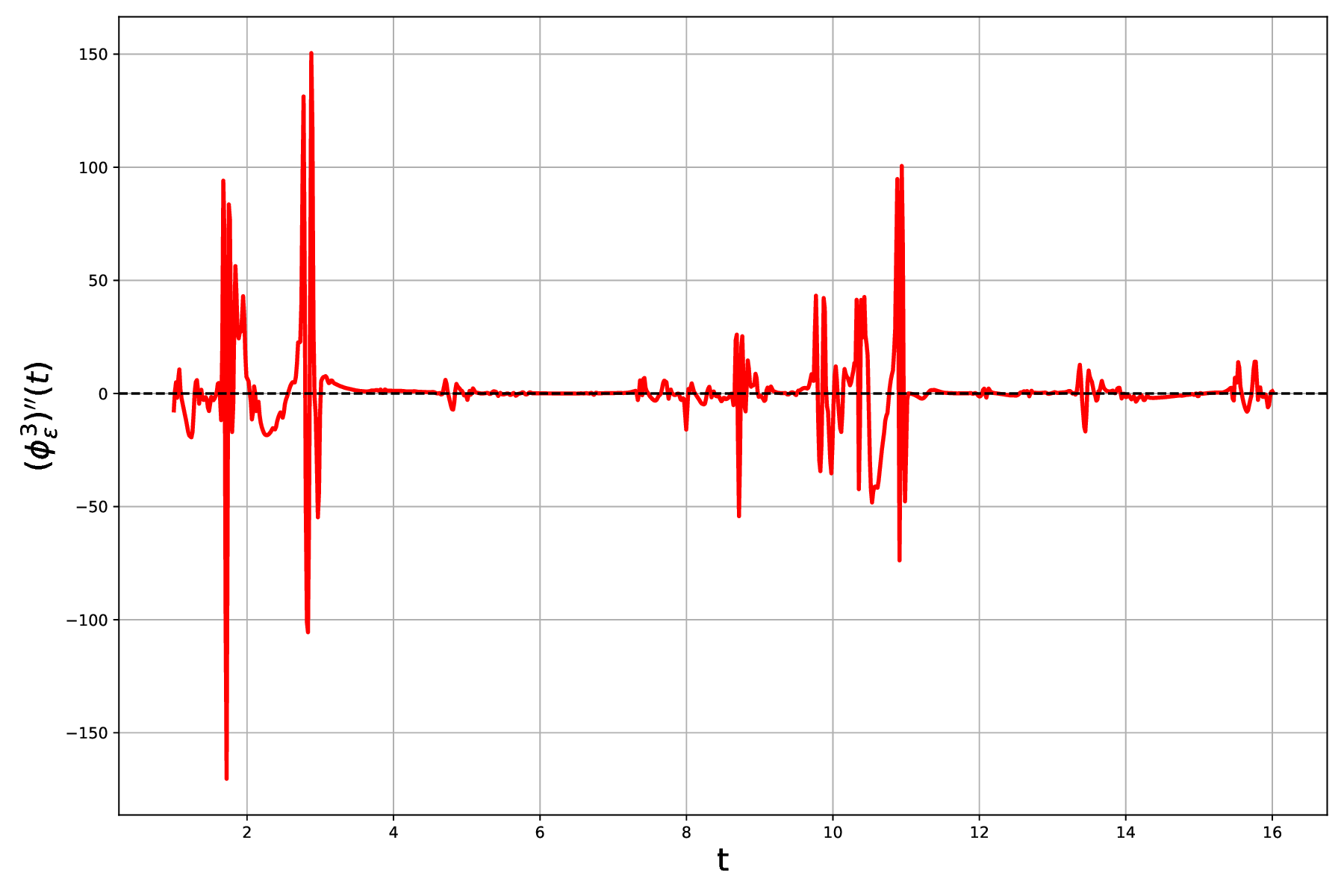}}\hfill
	\subfloat[Second derivative of $\phi_\epsilon^4$\label{fig:rctzfif_2d}]{\includegraphics[width=0.5\textwidth]{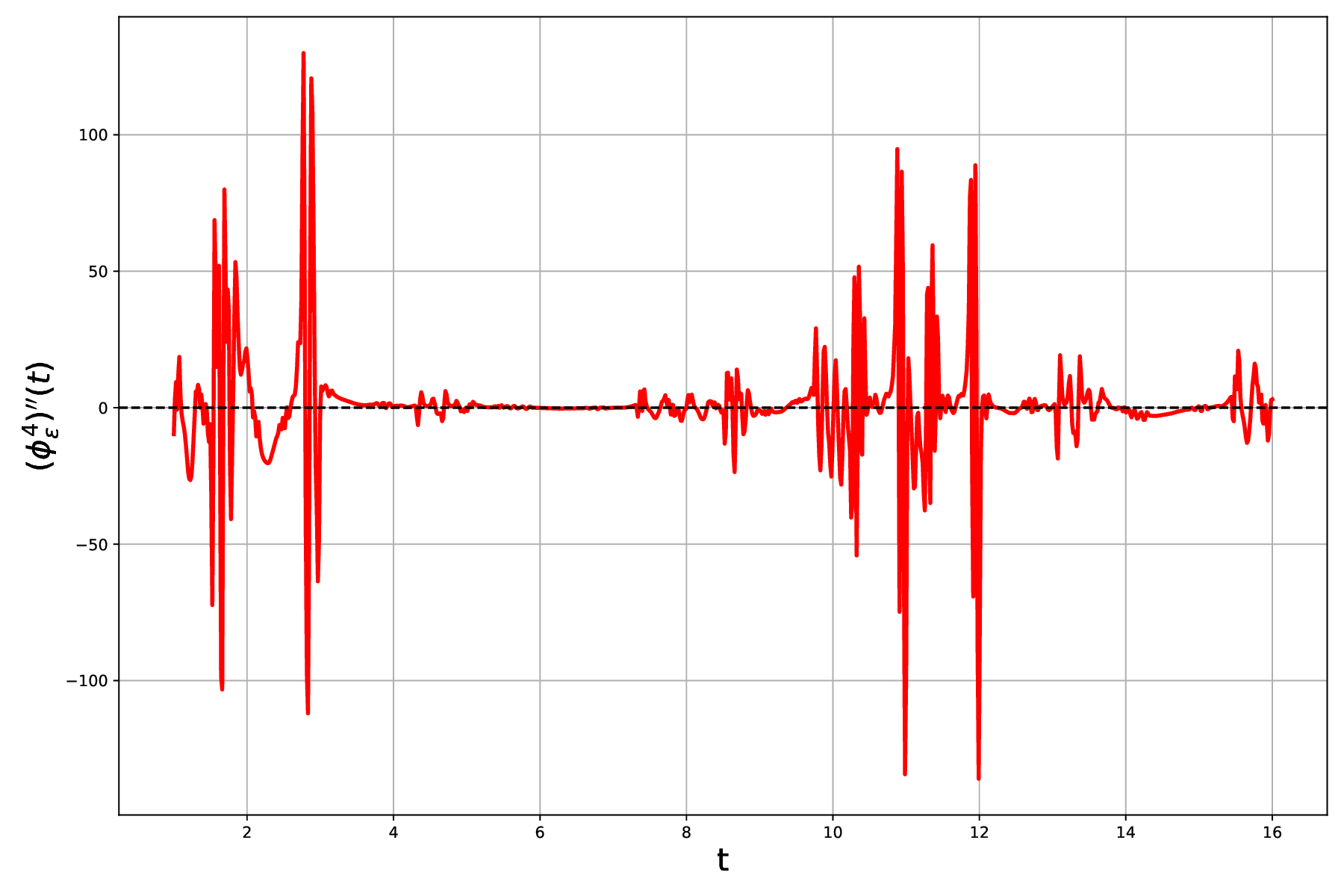}}\\
	\subfloat[Second derivative of $\psi_\epsilon^1$\label{fig:rctzfif_2e}]{\includegraphics[width=0.5\textwidth]{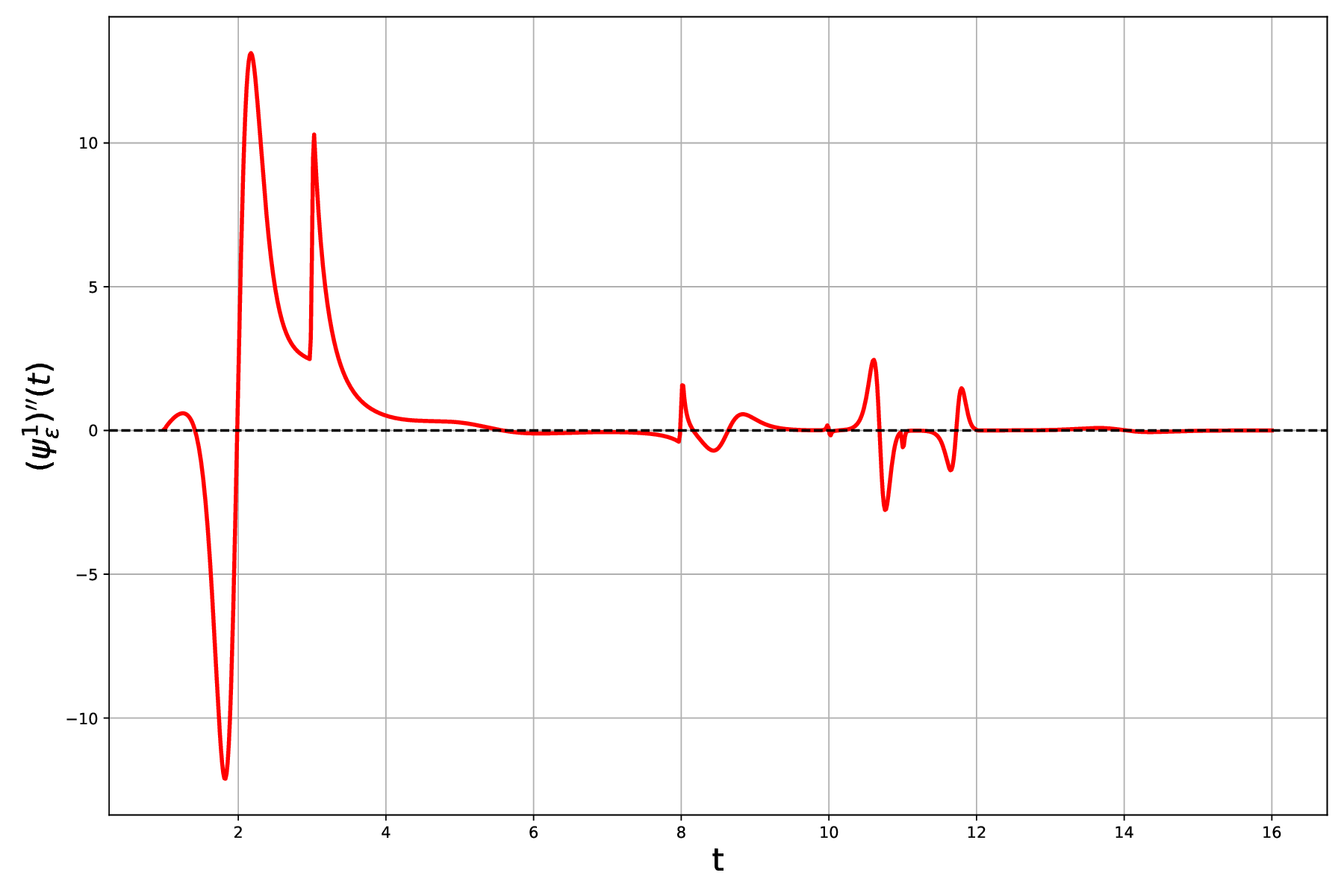}}\hfill
	\subfloat[Second derivative of $\psi_\epsilon^2$\label{fig:rctzfif_2f}]{\includegraphics[width=0.5\textwidth]{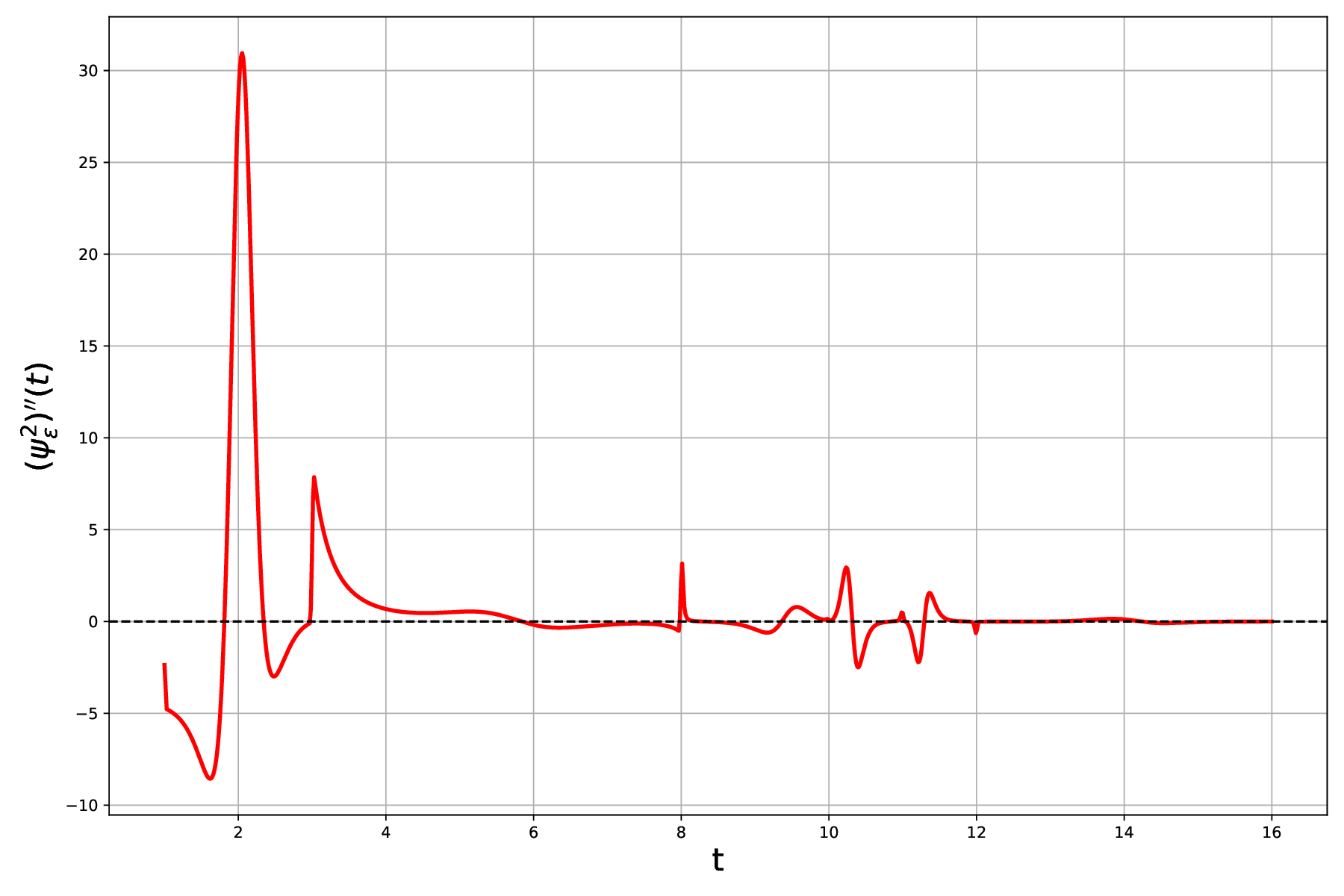}}
	\caption{Comparative visualization of second order derivatives of RCTZFIFs and their classical counterparts RCTZIFs from Fig. \ref{fig:rctzfif_all}.}
	\label{fig:2ndderrctzfif_all}
\end{figure}

Figure~\ref{fig:rctzfif_1a} displays an RCTZFIF generated without enforcing the positivity constraints from Theorem~\ref{thm:positivity}, resulting in non-positive values in the intervals $[3,8]$ and $[12,16]$. This underscores the necessity of adhering to the theoretical bounds. The prescribed bounds from Theorem~\ref{thm:positivity} for this dataset are $\lambda_1 \in [0, 0.1333]$,\ $\lambda_2 \in [0, 0.0571]$,\ $\lambda_3 \in [0, 0.0464]$,\ $\lambda_4 \in [0, 0.0535]$,\ $\lambda_5 \in [0, 0.05]$,\ and $\lambda_6 \in [0, 0.0492]$. Throughout our experiments, we fix $\alpha_{j,\epsilon} = 0.5$ and $\delta_{j,\epsilon} = 1$ across all subintervals. Figure~\ref{fig:rctzfif_1b} demonstrates a positivity-preserving RCTZFIF with parameters satisfying Theorem~\ref{thm:positivity}, maintaining non-negativity throughout $[1,16]$. Subsequent figures illustrate parameter sensitivity: Figure~\ref{fig:rctzfif_1c} shows localized effects of scaling parameter perturbations, while Figure~\ref{fig:rctzfif_1d} reveals how shape parameters $\beta_j$ and $\gamma_j$ influence local curvature. For benchmarking, Figures~\ref{fig:rctzfif_1e} and~\ref{fig:rctzfif_1f} present classical interpolants obtained by setting $\lambda_j = 0$ with different signature sequences. While these maintain positivity, they lack the fractal irregularity modeling capabilities of the full RCTZFIF framework.

\begin{table}[!htb]
	\centering
	\caption{Parameter configurations for RCTZFIF visualizations in Figure~\ref{fig:rctzfif_all}}
	\label{table:rctzfif_params}
	\begin{tabular}{|c|l|} \hline
		\textbf{Figure} & \textbf{IFS Parameters} \\ \hline
		\ref{fig:rctzfif_1a} & 
		$\lambda = [0.1323,0.2419,0.0561,0.0454,0.0526,0.149]$,\\
		& $\beta = [0.5028,1.1853,0.5,0.5,0.5,3.9649]$,\\
		& $\gamma = [0.5,0.5,0.5868,0.5221,0.5,0.5]$, 
		$\epsilon = [1,1,1,1,1,1]$ \\ \hline
		\ref{fig:rctzfif_1b} & 
		$\lambda = [0.1323,0.0201,0.0261,0.0454,0.0426,0.049]$,\\
		& $\beta = [0.5028,172.6956,6.5,0.5,22.5,0.5]$,\\
		& $\gamma = [0.5,5.5,0.5300,0.5221,0.5,0.5]$,  
		$\epsilon = [1,1,1,1,1,1]$ \\ \hline
		\ref{fig:rctzfif_1c} & 
		$\lambda = [0.1323,0.0201,0.0400,0.0454,0.0001,0.033]$,\\
		& $\beta = [0.5028,172.6956,6.5,0.5,22.5,0.5]$,\\
		& $\gamma = [0.5,5.5,0.5300,0.5221,0.5,0.5]$, 
		$\epsilon = [1,1,1,1,1,1]$ \\ \hline
		\ref{fig:rctzfif_1d} & 
		$\lambda = [0.1323,0.0201,0.0261,0.0454,0.0426,0.049]$,\\
		& $\beta = [0.5028,3.56,6.5,12.5,22.5,0.5]$,\\
		& $\gamma = [0.5,5.5,53,0.5221,0.5,0.5]$, 
		$\epsilon = [1,1,1,1,1,1]$ \\ \hline
		\ref{fig:rctzfif_1e} & 
		$\lambda = [0,0,0,0,0,0]$, 
		$\beta = [0.5028,3.56,6.5,12.5,22.5,0.5]$,\\
		& $\gamma = [0.5,5.5,53,0.5221,0.5,0.5]$, 
		$\epsilon = [0,0,0,0,0,0]$ \\ \hline
		\ref{fig:rctzfif_1f} &
		$\lambda = [0,0,0,0,0,0]$, 
		$\beta = [0.5028,3.56,6.5,12.5,22.5,0.5]$,\\
		& $\gamma = [0.5,5.5,53,0.5221,0.5,0.5]$, 
		$\epsilon = [1,1,1,1,1,1]$ \\ \hline
	\end{tabular}
\end{table}

	\begin{table}[h!]
	\centering
	\caption{Computational performance comparison (execution times in seconds)}
	\label{tab:computational_performance}
	\begin{tabular}{|l|c|c|c|c|c|}
		\hline
		Method & $N_{\text{plot}} = 37$ & $N_{\text{plot}} = 217$ & $N_{\text{plot}} = 1297$ & $N_{\text{plot}} = 7777$ & $N_{\text{plot}} = 46657$\\
		\hline
		Linear & 0.000253 & 0.001151 & 0.005161 & 0.029853  & 0.188199\\
		Cubic Spline & 0.000395 & 0.001173 & 0.006101 & 0.033135  & 0.188212\\
		PCHIP & 0.000374 & 0.001213 & 0.005662 & 0.033353  & 0.190643\\
		\textbf{RCTZFIF} & \textbf{1.4294} & \textbf{1.3479} & \textbf{1.7631} & \textbf{5.3084}  & \textbf{83.9312}\\
		\hline
	\end{tabular}
\end{table}

Table~\ref{tab:computational_performance} provides a systematic comparison of computational performance between RCTZFIF and classical interpolation methods for varying numbers of evaluation points ($N_{\text{plot}}$). The proposed method demonstrates a significant computational overhead, ranging from approximately 3,600$\times$ slower than cubic splines at $N_{\text{plot}} = 37$ to 446$\times$ slower at $N_{\text{plot}} = 46657$. This overhead is justified by several advanced computational requirements: iterative evaluation of fractal self-referential functional equations, computation of trigonometric rational basis functions, enforcement of positivity constraints, and management of $3(n-1)$ tunable parameters for precise shape control. Despite the constant-factor overhead, RCTZFIF exhibits favorable scaling characteristics. While classical methods show near-linear scaling, RCTZFIF maintains favorable $\mathcal{O}(n)$ scaling with evaluation points, ensuring computational tractability for practical scientific applications where mathematical guarantees outweigh raw computational speed.

Our systematic parameter tuning approach reveals distinct roles for each parameter class. Scaling parameters $\lambda_j$ predominantly control fractal complexity and local regularity, while shape parameters $\beta_{j,\epsilon}$, and $\gamma_{j,\epsilon}$ fine-tune local curvature and monotonicity behavior. The sensitivity analysis demonstrates that minor perturbations in these parameters significantly affect the roughness characteristics of the resulting interpolants, as evidenced by the second derivative plots. This parameter sensitivity creates a rich optimization landscape where evolutionary approaches show promise for identifying configurations that optimally balance interpolation accuracy with smoothness requirements.
	
	\section{Conclusion}
	\label{sec:conclusion}
	In this work, we introduced the RCTZFIFs, an interpolation framework that integrates the geometric flexibility of fractals with the structural control offered by rational cubic trigonometric functions. Our approach ensured the preservation of positivity in univariate datasets by imposing the suitable restrictions on the IFS parameters. Theoretical analysis confirmed the convergence of the RCTZFIF to to the original data-generating function. Our claims were validated through numerical experiments, where the effectiveness of our method was demonstrated. Overall, the RCTZFIF provides a powerful and flexible tool for modelling data with inherent geometric feature, positivity. Other shape preserving properties like monotonicity, convexity, etc., are under construction. Additionally, the extension of the proposed method to higher dimensional settings will be explored in future work.
	
	%
	%
	%
	%
	
\end{document}